\documentclass[11pt,reno]{amsart} 
\usepackage{amssymb}
\usepackage{amsmath}
\usepackage{enumitem}
\usepackage{mathrsfs}
\usepackage[english]{babel}
\usepackage[all]{xy}
\usepackage[hidelinks]{hyperref}
\usepackage{marginnote}
\usepackage{graphicx}

\usepackage{stmaryrd}

\usepackage[margin=1.28in]{geometry}

\RequirePackage{mathrsfs} 

\newtheorem{theorem}{Theorem}[section]
\newtheorem{lemma}[theorem]{Lemma}

\newtheorem{corollary}[theorem]{Corollary}
\newtheorem{conjecture_env}[theorem]{Conjecture} 
\newtheorem{alphtheorem}{Theorem}

\theoremstyle{definition}
\newtheorem*{ack}{Acknowledgements}
\newtheorem*{con}{Conventions}
\newtheorem{remark}[theorem]{Remark}

\newtheorem{definition}[theorem]{Definition}

\numberwithin{equation}{section} \numberwithin{figure}{section}

\DeclareMathOperator{\an}{an}

\DeclareMathOperator{\Hol}{Hol}

\DeclareMathOperator{\Hom}{Hom}

\newcommand{\bD}{{\bf\Delta}}

\newcommand\CC{\mathbb{C}}

\usepackage{color}

\definecolor{orange}{rgb}{1,0.5,0}

\title[Finiteness  of pointed maps to moduli spaces]{Finiteness of pointed maps to moduli spaces of polarized varieties}

\author{Ariyan Javanpeykar}
\address{PO Box 9010 \\ 
6500GL Nijmegen\\
The Netherlands.}
\email{ariyan.javanpeykar@ru.nl}

\author{Steven Lu}
\address{D\'epartement de math\'ematiques
Universit\'e du Qu\'ebec \`a Montr\'eal
Case postale 8888, succursale centre-ville
Montréal (Qu\'ebec) H3C 3P8}
\email{lu.steven@uqam.ca}

 \author{Ruiran Sun}
 \address{School of Mathematical Sciences, Xiamen University, Xiamen 361005, China}
\email{ruiransun@xmu.edu.cn}

\author{Kang Zuo}
 \address{ 
 School of Mathematics and Statistics, Wuhan University, Luojiashan, Wuchang, Wuhan, Hubei, 430072, P.R. China;
Institut für Mathematik, Universität Mainz, Mainz 55099, Germany.}
\email{zuok@uni-mainz.de}

\subjclass[2010]
{} 


\begin{document}

\begin{abstract}    
We establish a finiteness result for pointed maps to the base space $U$ of a smooth projective family of  varieties  with  maximal variation in moduli.  For its proof, we establish   the rigidity of pointed maps to  a (not necessarily compact)   variety which is hyperbolic modulo a proper closed subset. Together with Viehweg's hyperbolicity conjecture on the bigness of log-canonical bundles of moduli spaces, resolved by Campana--P\u{a}un, we derive an optimal dimension bound on the Hom scheme  from a curve to $U$ among other applications.
\end{abstract}

\maketitle

\thispagestyle{empty}

 \section{Introduction}   
  Compact complex hyperbolic varieties $X$ are known to enjoy many nice properties of which most are reflected in
  $\Hom(C,X)$ with $C$ a complex projective curve.  For example, from Theorem~5.3.9, Theorem~6.4.1.(2) and Corollary~6.6.11 of \cite{KobayashiBook}, one has that: \smallbreak
\begin{enumerate} 
\item   $\mathrm{Hom}(C,X)$ is \emph{compact} (i.e., boundedness holds).
\item  $\mathrm{Hom}(C,X)$ is hyperbolic (i.e., it \emph{inherits} its hyperbolicity from $X$). 
\item   The {dimension} of the subspace of non-constant morphisms $\mathrm{Hom}^{nc}(C,X)$ is strictly smaller than the dimension of $X$ (i.e., we have a \emph{dimension drop}).
\end{enumerate}
 Crucial to the proofs of (2) and (3) is the rigidity of pointed maps to $X$ (as proven in  \cite[Theorem~5.3.10]{KobayashiBook}), where we say that a pointed map $f\colon (C,c)\to (X,x)$ is rigid if  the connected component of the pointed Hom scheme $\mathrm{Hom}((C,c),(X,x))$ containing $f$ is a  singleton  (see Definition \ref{definition:rigidity} for a more precise definition).   \smallbreak

In this paper, we show these properties for a class of non-compact objects that have natural hyperbolicity properties and, in particular, for certain base spaces of families of (polarized) varieties. \smallbreak

 More specifically, we study these properties for the base space of a smooth projective family of     varieties with semi-ample canonical bundle which has maximal variation. In this setting,  hyperbolicity is well-known (Theorem \ref{thm:deng})  and   boundedness of Hom schemes as well   (Theorem \ref{thm:bounded}).   Our new results include a finiteness result for pointed maps (Theorem \ref{thm1}), the \emph{inheritance} property of hyperbolicity as well as the optimal \emph{dimension bound} or \emph{dimension drop} (Theorem \ref{thm:dimension_drop_intro}).  \smallbreak
 
  In fact,  our results for moduli spaces rely on a new result concerning (Kobayashi) hyperbolic varieties (see Theorem \ref{thm:nondeg_kob_implies_rel_compactness}): If $Z\subset U$ is a closed subset of an algebraic variety $U$ and $U$ is hyperbolic modulo $Z$, then every pointed map $(C,c)\to (U,u)$ with $u\not\in Z$ is rigid.  Note that this result extends Urata's rigidity theorem \cite{Urata} to the setting of (not necessarily compact) pseudohyperbolic varieties (going beyond the well-known extension to hyperbolically embeddable varieties).

 \subsection{Moduli spaces of varieties with semi-ample canonical bundle}    
In what follows, a variety will   be an irreducible, reduced quasi-projective scheme over $\mathbb{C}$.  \smallbreak

 Consider a  variety  $U$ parametrizing a ``maximally varying'' family of smooth projective varieties with semi-ample canonical bundle. 
  More precisely,  we assume throughout that there is a smooth projective morphism $\phi: V\to U$ whose  fibres $V_u$ (with $u\in U$) are projective varieties with semi-ample canonical bundle (hence  are good minimal models) such that the family $V\to U$ has \emph{maximal variation in moduli}, i.e., $\mathrm{Var}(f)=\dim U$ as defined in \cite[p.~1]{Vie83}, or equivalently, that the Kodaira-Spencer map $T_u U\to H^1(V_u, TV_u)$ is injective for a sufficiently general point  $u\in U$.    
 
 \begin{alphtheorem}\label{thm1}  With the assumptions as given above, there is a proper Zariski-closed subset $\Sigma  \subsetneq U$ such that, for every pointed variety $(Y,y)$ and  $u\in U\setminus \Sigma$, the set of pointed morphisms $f\colon (Y,y)\to (U,u)$ is finite.
 \end{alphtheorem}

Theorem \ref{thm1} is proven by combining the boundedness of the Hom scheme of all maps from $Y$ to $U$ due to Viehweg--Zuo (see Theorem \ref{thm:bounded} below) with a new rigidity theorem for pointed maps (see Section \ref{section:final}).

 Note that one cannot expect $\Sigma$ in general to be empty in the theorem.  The hypothesis being invariant under proper modifications of $U$, one has for example to avoid maps into the exceptional locus of the blowup of a smooth point of $U$ for the conclusion to hold.   \smallbreak

In addition, the restriction on $f$ to be pointed cannot be lifted. Indeed, 
the Hom scheme of maps from a smooth curve $C$ to $U$ not factoring through $\Sigma$  is   
rarely zero-dimensional (i.e.,   rigidity rarely holds).    For example,  for $U$ a suitable curve, there are non-rigid non-isotrivial  polarized families of abelian varieties (even without isotrivial factors) \cite{FaltingsArakelov}, K3 surfaces
\cite[\S 6]{SaitoZucker}, and canonically polarized surfaces \cite[Example~1.13]{KovacsStrong}.    \smallbreak

 When   $U$   is as in Theorem \ref{thm1},   combining    the main result of \cite{DengLuSunZuo} with Theorem \ref{thm1} yields  a  proper closed subset $\Sigma' \subsetneq  U$ such that, for every pointed variety $(Y,y)$ and every $u$ in $U\setminus \Sigma'$, the set of \emph{holomorphic} maps $f\colon Y^{}\to U^{}$ with $f(y) = u$ is finite.   \smallbreak

If  the fibres of the family $V\to U$ satisfy infinitesimal Torelli, then
Theorem \ref{thm1} can alternatively be deduced from properties of period maps -- specifically  Deligne's finiteness result for monodromy representations and the Theorem of the Fixed Part; see  \cite[Theorem~1.7]{JLitt}.   In particular, the statement of   Theorem \ref{thm1} is not new if the fibres of  $V\to U$ are Calabi-Yau  varieties or abelian varieties.   \smallbreak

  Theorem \ref{thm1} is inspired by standard hyperbolicity properties suggesting that well-known moduli spaces of varieties should satisfy a pointed version of the Shafarevich conjecture (cf. \cite[Conjecture~1.5]{JSZ}).
We mention that in \cite{JSZ}, finiteness results are proved when sufficiently many base points on the source are fixed.  However, the approach there  
is completely different from here: Our one pointed finiteness result here is a consequence of a  general rigidity property for   pseudohyperbolic varieties (see Theorem \ref{thm:geometric_hyperbolicity_intro} or Theorem \ref{thm:from_kob_to_pointed_rigidity}).  In \cite{JSZ}, existing Viehweg--Zuo-type results and constructions are  adapted to the setting of $N$-pointed maps and hence a large part of \cite{JSZ} focuses on the construction of suitable Viehweg--Zuo sheaves.   \smallbreak

 Another motivation for this paper comes from the following semipositivity conjecture of Viehweg and Zuo (\textit{cf.} \cite[Problem~1.5]{VZ02}): Let $D$ be a simple normal crossings divisor on a smooth projective variety $X$.  If $\Sigma\subset X$ is a closed subset, we say that the pair $(X,D)$ is \emph{semi-positive modulo $\Sigma$} if, for every smooth projective curve $C$ and morphism 
$\bar{\gamma}\colon \bar{C}\to X$ such that $\bar{\gamma}(\bar{C})\not\subset \Sigma$,  any quotient sheaf of $ \bar{\gamma}^*\Omega^1_X(\log D) $ has non-negative degree.

\begin{conjecture_env}[Viehweg--Zuo]  \label{conj:vz} 
Let $X$ be a smooth projective variety, $D$ be a simple normal crossings divisor and $U = X \setminus D$.
If $U$ is the base space of a smooth  maximally varying family of projective varieties with semi-ample canonical bundle then there is a proper closed subset $\Sigma\subsetneq X$ containing $D$ such that $(X,D) $ is semi-positive modulo $\Sigma$.
\end{conjecture_env}

We note that Conjecture~\ref{conj:vz} implies the rigidity of pointed maps, so that our main result (Theorem~\ref{thm1}) can be viewed as supporting
evidence for Conjecture~\ref{conj:vz}. Indeed, with the notation as given, a non-rigid pointed morphism $\gamma$ from a pointed curve $(C,c)$ to $(U,u)$ would yield 
a nonzero element of
\[
\mathrm{H}^0\left(\bar{C},  \mathcal{O}_{\bar{C}}(-c) \otimes \bar{\gamma}^* T_X(-\log D)\right)
= \mathrm{Hom}\left(\mathcal{O}_{\bar{C}}(c), \bar{\gamma}^* T_X(-\log D)\right),
\]
where $\bar{\gamma}\colon \bar{C}\to X$ is the extension of $\gamma$ to a smooth compactification $\bar{C}$ of $C$, contradicting the semipositivity of $(\bar{\gamma})^\ast\Omega^1_X(\log D)$. \smallbreak

Although Conjecture \ref{conj:vz} remains open (even when $D=\emptyset$),  it is supported by Campana and P\u{a}un's complete solution to \emph{Viehweg's hyperbolicity conjecture} (see Theorem \ref{thm:CP}). Indeed, in \cite{CP} they proved the pseudo-effectivity of  $\Omega^1_X(\log D)$. \smallbreak

Conjecture \ref{conj:vz} is known if the fibres of $V\to U$ satisfy infinitesimal Torelli (see \cite[Theorem~0.1.(i)]{ZuoNeg}); by \emph{loc. cit.}, if $U$ admits a generically immersive period map, then $(X,D)$ is semi-positive modulo some proper closed subset of $X$.

 \subsection{Boundedness}  Although our main contribution in this paper is the rigidity of certain pointed morphisms, we will naturally need to invoke  the ``boundedness'' of certain families of varieties to deduce our finiteness result (Theorem \ref{thm1}); we refer the reader to Section \ref{section:hom_scheme} for the definition of the Hom schemes discussed below.     
 
   \begin{definition}
   Let $U$ be a (quasi-projective) variety.  Let $\Sigma\subset U$ be a Zariski closed subset. We say that $U$ is \emph{bounded modulo $\Sigma$} if, for every smooth quasi-projective curve $C$,   
   \[{\Hom}(C,U)\setminus {\Hom}(C,\Sigma)\] is of finite type (hence quasi-projective) over $\CC$.  \end{definition}
   \begin{definition}\label{def:bounded}  A variety $U$ is \emph{pseudo-bounded} (resp. \emph{bounded}) if there exists a proper closed subset $\Sigma\subsetneq U$ such that $U$ is bounded modulo $\Sigma$ (resp.  bounded   modulo $\emptyset$). 
   \end{definition}

   The notion of  boundedness in the setting that $U=X$ is a projective variety coincides with the earlier notion of weak boundedness introduced by Kov\' acs-Lieblich  \cite{KovacsLieblich} (cf. \cite{JKa} and \cite{vBJK}). Hence, if $U=X$ is projective and Brody hyperbolic, then $U$ is   hyperbolic and bounded (see \cite[\S 5]{KobayashiBook}).  Such varieties arise naturally in the study of families of polarized varieties.
Indeed,  let $U$ be a variety and suppose that there exists a family $ V \to U$ of smooth projective varieties with ample canonical bundle such that, for every variety $V_0$, the set of $u$ in $U$ with $V_u\cong V_0$ is finite.  

Then \cite{VZ03} shows that the variety $U$ is Brody hyperbolic and thus bounded if $U$ is projective. 
Even if $U$ is not projective, it follows from the following theorem of   Viehweg--Zuo that $U$ is bounded \cite[Theorem~6.2.(ii)]{VZ02}   (see also \cite[Theorem~10.8]{Kovacs}, \cite{KovacsLieblich}, and \cite{AscherTaji}   for related  results).

\begin{theorem}[Viehweg--Zuo, 2002]\label{thm:bounded}  Let $U$ be a   variety. Assume that there exists 
  a smooth projective family $f\colon V\to U$ of varieties with semi-ample canonical bundle which has  maximal variation in moduli.    
 Then $U$ is pseudo-bounded.
\end{theorem}

\subsection{A criterion for the rigidity of pointed maps} The key to Theorem \ref{thm1}  is the rigidity of pointed maps  $(C,c)\to (U,u)$ for $C$ a smooth curve. \smallbreak 

Consider the case when the base space $U$ is compact in Theorem \ref{thm1}.  One can appeal to   a well-known criterion for rigidity of  pointed  maps: if $Y$ and  $U$ are proper varieties,  then a pointed map $f\colon(Y,y)\to (U,u)$ is rigid   if  (and only if) the connected component of ${\Hom}((Y,y),(U,u))$  containing $f$ is proper, cf. \cite[\S 4, p. 43]{MumAb}. This is a part of Mori's bend-and-break, see for example \cite[\S 3]{Debarrebook}, which says that this fails to be the case if and only if there is a rational curve in $U$ containing $u$, i.e. a nonconstant morphism ${\mathbb P}^1\to U$. In particular, from well-known hyperbolicity results on our moduli spaces considered, all rational curves are contained in some proper closed subset of $U$ in our case of maximal variation, so that 
our rigidity theorem readily follows when  $Y$ and $U$ are compact. 
 \smallbreak

However, most base spaces of families of smooth varieties, such as $U$, are not proper. When attempting to generalize this classical rigidity criterion to the quasi-projective case, one quickly encounters difficulties in controlling how rational curves, arising from bend-and-break, intersect the boundary of the moduli space. This challenge reminds us of that of the Log Bend-and-Break Conjecture of Keel-McKernan \cite[Conjecture~1.11]{KeelMcKernan}, which remains unresolved for dimensions greater than two. The difficulties for the related Conjecture~\ref{conj:vz} are also in a sense quite similar to those in Campana--P\u{a}un's proof of Viehweg's hyperbolicity conjecture.  Indeed, in the case that the base space $U$ is proper, the pseudo-effectivity of the  cotangent bundle of a moduli space can be established using Miyaoka's generic semi-positivity theorem combined with the hyperbolicity result of Viehweg--Zuo (see  Patakfalvi \cite{Patakfalvi}). However, to extend this result to the non-proper setting, Campana--P\u{a}un resorted to Bogomolov-McQuillan's criterion for the algebraicity of foliations as well as an orbifold version of Viehweg's weak positivity theorem for direct image sheaves; see \cite{CP}. \smallbreak

When $U$ is not necessarily proper, our   idea in obtaining the rigidity of pointed maps $(C,c)\to (U,u)$ with $C$ a curve   is to restrict to a sufficiently small   Euclidean neighbourhood $\Delta$ of $c$ in $C$.   
This allows us to regard the Hom scheme $\Hom((C,c),(U,u))$ as a subspace of the space of holomorphic maps $\mathrm{Hol}(\Delta,U^{})$ (or even of $\mathrm{Hol}((\Delta,c),(U^{},u))$). Then,   $\Hom((C,c),(U,u))$ is easily seen to be zero-dimensional if $\Hom((C,c),(U,u))$ is relatively compact in $\mathrm{Hol}(\Delta,U^{})$.

\begin{alphtheorem}  \label{thm:rig_via_ac}
Let $(U,u)$ be a pointed variety, $(C,c)$ a smooth  quasi-projective pointed curve. If there is a Euclidean open neighbourhood $\Delta\subset C^{}$ of $c$ such that   \[\mathrm{Hom}((C,c),(U,u))\]  is  \textbf{relatively compact} in $\mathrm{Hol}(\Delta, U^{})$ with the compact-open topology, then every pointed morphism $(C,c)\to (U,u)$ is rigid, i.e., $\mathrm{Hom}((C,c),(U,u))$ is zero-dimensional.
\end{alphtheorem}

  It turns out that   relative compactness of the space of pointed morphisms  after restriction to a small enough neighbourhood is intimately related to the nondegeneracy of the Kobayashi pseudometric.  Although the  idea of restricting to smaller discs is classical in hyperbolic geometry, it was recently used in Yamanoi's seminar paper \cite{YamanoiKob}  to prove a  key technical  relative compactness lemma of a certain family of maps from a given algebraic curve to a general type subvariety of  an abelian variety.

 \subsection{From hyperbolicity to pointed rigidity}

Following Kobayashi \cite[\S 3.2]{KobayashiBook}, we recall that the \emph{Kobayashi pseudometric} $d_U$ on a complex space $U$ is the largest pseudometric on $U$ satisfying the distance-decreasing property with respect to holomorphic maps from the unit disc 
$
\bD \subset \mathbb{C}
$
equipped with the Poincaré distance $d_\bD$, classically defined as the path-integrated form of the Kähler metric norm
$
\frac{|dz|}{1 - |z|^2}.
$  We say that $U$ is \emph{hyperbolic} if $d_U$ is a metric on $U$. More generally,   for $Z\subset U$ a subset, we say that $U$ is \emph{hyperbolic modulo $Z$} if $d_U$ is nondegenerate on $U\setminus Z$.  We say that $U$ is \emph{pseudohyperbolic} if $U$ is hyperbolic modulo a proper closed subset $Z\subsetneq U$. 

We   prove that Kobayashi's notion of hyperbolicity implies the relative compactness of moduli spaces of pointed maps.

\begin{alphtheorem} \label{thm:almost_compactness} 
Let $U$ be a variety and let $\Sigma\subsetneq U$ be a proper closed subset such that $U$ is hyperbolic modulo $\Sigma$.  Let $u\in U\setminus \Sigma$. 
 Then, for any smooth quasi-projective pointed curve $(C,c)$ and $u\in U\setminus \Sigma$, there is a Euclidean open neighbourhood  $\Delta\subset C$  of  $c$ such that $\mathrm{Hom}((C,c),(U,u))$ is relatively compact in the space $\mathrm{Hol}(\Delta, U^{})$ of holomorphic maps $\Delta \to U^{}$ endowed with the compact-open topology.
\end{alphtheorem}

Theorem \ref{thm:almost_compactness} is proven in Section \ref{section:compactness} and  arguably forms the technical core of this paper.  In fact,  we prove more precise results showing that the nondegeneracy of the Kobayashi pseudometric near a point  $u$ is \emph{equivalent} to the relative compactness of the moduli space of pointed maps from small enough discs (formulated appropriately); see Theorem \ref{thm:equivs}    for a precise statement.  

Combining Theorem \ref{thm:rig_via_ac} and Theorem  \ref{thm:almost_compactness}, we obtain the rigidity of pointed maps to a pseudohyperbolic variety (see Theorem \ref{thm:from_kob_to_pointed_rigidity}).  A standard slicing argument (Lemma \ref{lem:geom_hyp_on_curves}) then leads to   the following useful criterion for   the  finiteness of pointed maps:

   \begin{alphtheorem} \label{thm:geometric_hyperbolicity_intro} Let $U$ be a variety and let $Z\subset U$ be a closed subset.  Suppose that $U$ is hyperbolic modulo $Z$ and bounded modulo $Z$.  Then, for every pointed variety $(Y,y)$ and $u\in U\setminus Z$, the set of pointed morphisms $(Y,y)\to (U,u)$ is finite.  
   \end{alphtheorem}

Theorem \ref{thm1} is  proven by combining  Theorem \ref{thm:geometric_hyperbolicity_intro} with the boundedness result of Viehweg--Zuo (Theorem \ref{thm:bounded}) and the fact that $U$ (as in Theorem \ref{thm1}) is pseudohyperbolic (see Theorem \ref{thm:deng}). 
  
 \subsection{Optimal dimension bounds}

We conclude the  paper  with the aforementioned applications to the inheritance property of hyperbolicity and to the optimal  dimension bounds.    Let $P^c$ be the coarse moduli space of Viehweg's moduli functor \cite[\S 7.6]{ViehwegBook}.

\begin{alphtheorem} \label{thm:dimension_drop_intro}    
  Let $U$ be a variety. Assume that there exists a smooth projective    family     $V\to U$ of varieties with semi-ample canonical bundle and a relatively ample line bundle $L$ on $V$ such that the moduli map  $U\to P^c$ associated to the pair $(V,L)$  is quasi-finite.    Let $C$ be a  smooth quasi-projective curve. Then, the following statements hold.
  \begin{enumerate}
  \item  For every smooth quasi-projective curve $C$, the Hom scheme   $\mathrm{Hom}^{nc}(C,U)$ of non-constant maps satisfies
\[
\dim  \mathrm{Hom}^{nc}(C,U)  \leq \dim U - 1, 
\]  
\item   every subvariety of $\mathrm{Hom}(C,U)$  is   pseudohyperbolic,  and
\item    every subvariety of $\mathrm{Hom}(C,U)$  is  of log-general type.
\end{enumerate}       
\end{alphtheorem}

This dimension drop statement in Theorem \ref{thm:dimension_drop_intro} is a consequence of a more general result pertaining to pseudohyperbolic   varieties of log-general type; see Theorem \ref{thm:dimension0b}.   
   Moreover,  this bound  can not be improved.  For example,  if  $U$ is a product of curves $C\times \ldots \times C = C^{\dim U}$, then  $\mathrm{Hom}^{nc}(C,U)$ contains $C^{\dim U-1}$.  On the other hand,  if $C$ is the base space of a non-isotrivial smooth proper family of curves of genus at least two, then  $U$ is the base space of a family of canonically polarized varieties with quasi-finite moduli map.

We note that the dimension drop result (Theorem \ref{thm:dimension_drop_intro}) plays a crucial role in recent work    on (not necessarily pointed) non-rigid subvarieties of the moduli space of canonically polarized varieties; see   \cite[Section~6]{Chenetal}.

\raggedbottom

 \begin{con} The base field is $\CC$ throughout. A variety is a separated irreducible reduced quasi-projective scheme.  A point on a variety is a closed point of the scheme.   If $V$ is a locally finite type scheme over $\mathbb{C}$, we also write $V$ for its  associated complex space $V^{an}$ \cite[\S 12]{SGA1} unless emphasis is needed.
  \end{con}

 \begin{ack}
 The first-named author is grateful to Kenneth Ascher and Behrouz Taji for helpful discussions on boundedness and to Stefan Kebekus for  useful discussions.
 \end{ack}

 \section{Hom scheme of maps}  \label{section:hom_scheme}
Let $\overline{C}$ be a smooth projective curve,  and let $X$ be a  projective variety with an ample line bundle  $L$. Then, the set $\Hom^d(\overline{C},X)$ of morphisms $f\colon\overline{C}\to X$ with $\deg f^\ast L \leq d$ is (the set of closed points on) a quasi-projective scheme; see \cite[\S2]{Debarrebook}. This scheme can be non-reduced and non-irreducible.   Note that  $\Hom(\overline{C},X)=\sqcup_{d\geq 0} \Hom^d(\overline{C},X)$ is a scheme.

Let $C\subset \overline{C}$ and $U\subset X$ be Zariski opens. Write   $D=X \setminus U$.
As every morphism $f\colon C\to U$ extends uniquely  to a morphism $\overline{f}\colon \overline{C}\to X$ with $\overline{f}(C)\not\subset D$, the set $\mathrm{Hom}(C,U)$ of morphisms $f\colon C\to U$ is a  subset of (the closed points of) the scheme ${\Hom}(\overline{C},X)  \setminus {\Hom}(\overline{C},D) $.  If $d$ is an integer, we define $\Hom^d(C,U) := \Hom(C,U) \cap {\Hom}^d(\overline{C}, X)$.
 
 We note that ${\Hom}(\overline{C},D)\subset {\Hom}(\overline{C},X)$ is a closed subset (cf. \cite[\S3]{BJR}). In particular, its  complement ${\Hom}(\overline{C},X)\setminus {\Hom}(\overline{C},D)$ is   open and   the subset   \[\Hom(C,U) \subset {\Hom}(\overline{C},X)\setminus {\Hom}(\overline{C},D)\] is closed. In particular, for every integer $d$, the scheme     ${\Hom}^d(C,U)$ is  quasi-projective.   Hence  ${\Hom}(C,U) = \sqcup_{d\geq 0} {\Hom}^d(C,U)$  is closed in ${\Hom}(\overline{C},X)\setminus {\Hom}(\overline{C},D)$.
   Since the evaluation map $\mathrm{ev}\colon C\times {\Hom}(C,U) \to U$ is a morphism of schemes (hence continuous), for every $c$ in $C$ and $u$ in $U$, the subset $\mathrm{Hom}((C,c),(U,u)) \subset {\Hom}(C,U)$ is closed.   
   
\begin{definition}\label{definition:rigidity} 
 If $f\colon C\to U$ is a morphism with $f(c) = u$,  we let $H_f$ be the (unique)  connected component of  $\Hom((C,c),(U,u))$ containing the point corresponding to $f$.  We say that $f$ is \emph{($1$-pointed) rigid} if $H_f$ is a point.
\end{definition}
 For example,  if ${\Hom}((C,c),(U,u))$ is zero-dimensional,  then every   morphism $(C,c)\to (U,u)$ is rigid.  Moreover, if $U = X$ is projective, then it is well-known that  the following are equivalent (cf.  \cite[Lemma~3.5 and Proposition~3.12]{JKa}). 
 \begin{enumerate}
 \item The projective variety $X$ has no rational curves.
 \item For every smooth projective curve $C$, every   $c\in C$ and every   $x\in X$,   the scheme $\Hom((C,c),(U,u))$ is zero-dimensional.
 \end{enumerate}

 \section{A variant of the rigidity lemma}\label{section:rigidity}
 Let $X$, $Y$ and $Z$ be algebraic varieties (resp. complex  spaces) and let $x_0\in X$.  {Let $X\times Y\to Z $ be a morphism that contracts $\{x_0\}\times Y\to Z$ to a point of $Z$.} The Rigidity Lemma then implies that every fibre $\{x\}\times Y$ is contracted, provided $Y$ is proper (resp. compact); see  \cite[\S 4, p.~43]{MumAb}  (resp.  \cite[Lemma~5.3.1]{KobayashiBook}). The Rigidity Lemma  relies on the fact that regular functions on a proper algebraic variety are constant (resp. that holomorphic functions on a compact connected complex space are constant).Here, we provide a variant of the Rigidity Lemma via the elementary fact that bounded holomorphic functions on an \emph{algebraic} variety are constant \cite[Remark~2.9]{JV}.

  \begin{theorem}[Variant of the Rigidity Lemma]\label{thm:rigidity_lemma}
 Let $(T,t_0), (U,u_0)$ be pointed connected complex   spaces.  Let $Y$ be a complex space on which every bounded holomorphic function is constant (e.g., a  complex algebraic variety).  Let $\Phi\colon T\times Y^{}\to U^{}$ be a morphism that contracts $\{t_0\}\times {Y}$ to $\{u_0\}$ and admits a continuous extension $\bar \Phi\colon T\times \bar Y\to U$   for a  (topological) compactification $\bar Y$ of $Y$.
 Then, for each $t$ in $T$,  the morphism $\Phi$ contracts $\{t\} \times Y$ to a point.
 \end{theorem} 
\begin{proof}  
Let $S := \{ t \in T \mid \Phi(\{t\} \times Y) \text{ is a point} \}$. By assumption, $t_0 \in S$, so $S$ is nonempty.
 f $t_i \in S$ and $t_i \to t$ in $T$, then the restriction of $\bar{\Phi}$ to $\{t_i\} \times \bar{Y}$ is constant (since it factors through a point). By continuity of $\bar{\Phi}$, the limit $\bar{\Phi}|_{\{t\} \times \bar{Y}}$ is also constant, so $\Phi(\{t\} \times Y)$ is a point. Hence $S$ is closed.

To conclude the proof, it suffices to  show that $S$ is open. Let $t \in S$, and choose a Stein neighborhood $U' \subset U$ of the point $u := \bar{\Phi}(\{t\} \times \bar{Y})$. By continuity, there is a neighborhood $T' \subset T$ of $t$ such that $\bar{\Phi}(T' \times \bar{Y}) \subset U'$. For any $t' \in T'$ and any $f \in \mathcal{O}(U')$, the composition $f \circ \Phi|_{\{t'\} \times Y}$ is a bounded holomorphic function on $Y$, hence constant. Since holomorphic functions separate points on $U'$, it follows that $\Phi(\{t'\} \times Y)$ is a single point. Hence $t' \in S$, so $S$ is open.
Thus, as $T$ is connected and $S$ is nonempty, open, and closed, we conclude $S = T$.
\end{proof}

 We now prove Theorem \ref{thm:rig_via_ac} by showing that  an \emph{algebraic} family of pointed maps $(C,c)\to (U,u)$ which, when restricted to any open neighbourhood $\Delta\subset C^{\an}$ is relatively compact in the space of holomorphic maps from $\Delta$ to $U^{\an}$,  is zero-dimensional.

\begin{theorem} \label{thm:rig_via_ac0}  
Let $(U,u)$ be a pointed variety and let $ (C,c)$ be  a smooth quasi-projective pointed curve. Let $H\subset \mathrm{Hom}((C,c),(U,u))$ be a locally closed subset. 
If $\Delta\subset C^{\an}$ is an open neighbourhood containing $c$ such that $H$ is  \textbf{relatively compact} in the set of holomorphic maps $\mathrm{Hol}(\Delta, U^{\an})$ with the compact-open topology, then  $H$ is zero-dimensional.
\end{theorem}
\begin{proof}  
Let $Y\subset  H$ be an  irreducible   component  and let   $\bar{Y}$ be its   closure     in $\mathrm{Hol}((\Delta,c),(U^{},u))$ with the compact-open topology.  Note that, by assumption,    the topological space $\bar{Y}$ is compact.

The evaluation map $\mathrm{ev}\colon C \times \mathrm{Hom}((C,c),(U,u)) \to U$   is a morphism of schemes and its analytification $\mathrm{ev}^{\an}\colon C^{\an} \times \mathrm{Hom}((C,c),(U,u))^{\an} \to U^{an}$ is holomorphic. The map from $\mathrm{Hom}((C,c),(U,u))^{\an}$ to $\mathrm{Hol}((\Delta,c),(U,u))$ given by restriction is continuous. Consequently, the restricted evaluation   map $\Phi\colon \Delta \times \bar{Y} \to U^{\an}$ is continuous.  (Here we invoke the following  basic property of the compact-open topology: if $A,B,C$ are topological spaces with $B$ locally compact Hausdorff, then  composition $\mathrm{Con}(B,C) \times \mathrm{Con}(A,B) \to \mathrm{Con}(A,C)$ is continuous.) Note that  $\Phi$ extends $\Delta \times Y \to U$ and contracts $\{c\} \times \bar{Y}$ to $\{u\}$.

Since $Y$ is an algebraic variety,  we may apply the   above variant of the Rigidity Lemma (Theorem \ref{thm:rigidity_lemma})   to $\Phi$ and conclude that, for every $x$ in $\Delta$, the subset $\{x\}\times Y^{}$ is contracted to a point in $U^{}$.  In other words, if $f$ and $g$ are elements of $Y$, then $f(x) = g(x)$ for every $x$ in $\Delta$. This implies that $f=g$ and shows that $Y$ is a point.
\end{proof}

  \section{Relative compactness and rigidity of family of pointed maps} \label{section:compactness}
      Recall that, for $U$ a complex space, we let $d_U$ denote the Kobayashi pseudometric.
For a given point $u$ in $U$, we will relate the   rigidity of pointed maps to $U$ to the nondegeneracy of the Kobayashi pseudometric near the point $u$:

\begin{definition}
We say that $U$ is \emph{hyperbolic at $u$} if $d_U$ is nondegenerate on an open neighborhood of $u$. 
\end{definition}

The following result is our main observation concerning the Kobayashi pseudometric $d_U$. Informally,, it relates the nondegeneracy of $d_U$ at a point $u$ to the relative compactness of the space of pointed maps $(\bD,0)\to (U,u)$ after possibly restricting to a small enough open neighbourhood of $0$ in $\bD$.  This observation and its proof are essentially classical, though we are not aware of a direct reference in our specific setting.

 \begin{theorem}\label{thm:nondeg_kob_implies_rel_compactness}
 Let $U$ be a complex space and let $u\in U$. Suppose that $U$ is hyperbolic at $u$.  Then, if $\bD$ is the open unit disc,   there is a neighbourhood $\bD_0 \subset \bD$ of $0$ such that the natural inclusion
    \[
    \Hol((\bD, 0), (U, u)) \longrightarrow \Hol(\bD_0, U), \quad f \mapsto f|_{\bD_0}
    \]
    maps $\Hol((\bD, 0), (U, u))$ to a relatively compact subset of $\Hol(\bD_0, U)$, the latter space endowed with the compact-open topology.
 \end{theorem}
 \begin{proof}   
 
By the triangle inequality (for $d_U$),  the pseudometric $d_U$ is a continuous function. Let $V$ be an open neighbourhood of $u$ such that $d_U$ is nondegenerate on $V$.
  Since $U$ is locally compact (by the definition of a complex space),  the restriction $d_U|_V$ is a genuine metric on $V$ and induces the topology of $V$; see \cite[Theorem~3.1.15]{KobayashiBook} and \cite[Theorem~1.1.8.(1)]{KobayashiBook} for these basic facts.

 Since $d_U|_V$ is a metric on an open neighbourhood $V$ of $u$,  for a sufficiently small $r>0$,  there is a $d_U$-ball neighbourhood $W = B_{d_U}^r(u)$ of radius $r > 0$ centered at $u$ that is     relatively compact in $V$. Let $\bD_0:=\bD^{\rho}_r$ be the Poincaré disc of $d_\bD$-radius $r$ around $0 $ in $\bD$. 
As any holomorphic map $f : (\bD, 0) \to (U, u)$ is $d_U$-distance decreasing, we have $f(\bD^{\rho}_r) \subset W$. 
In particular, every $f \in \mathrm{Hol}((\bD, 0), (U, u))$ restricts to a holomorphic map $f|_{\bD_\rho^r} \colon \bD_\rho^r \to U$ with image in $W$ and, by the distance decreasing property, the family of maps
\[
\mathcal{F} := \left\{ f|_{\bD_\rho^r} \colon \bD_\rho^r \to U \;\middle|\; f \in \mathrm{Hol}((\bD, 0), (U, u)) \right\}
\]
is equicontinuous with respect to $d_U|_V$.

Note that, for every $x$ in $  \bD^{\rho}_r$, the set of images $\{f(x) \ | \ f\in \mathcal{F}\}$ is relatively compact in $U$ (as it is contained in $W$). Thus, by applying the Arzelà--Ascoli theorem with respect to the metric $d_U|_V$ on $V$, the family $\mathcal{F}$ is relatively compact in   $\Hol(\bD^{\rho}_r, V)$, and thus in $ \Hol(\bD^{\rho}_r, U) = \Hol(\bD_0,U)$, as required.  
 \end{proof}

\begin{proof}[Proof of Theorem \ref{thm:almost_compactness}]  
 By our assumption on $U$ and $Z$, the Kobayashi pseudometric $d_U$ is non-degenerate at $u$.  In particular,    by Theorem  \ref{thm:nondeg_kob_implies_rel_compactness},  after choosing a Euclidean open neighborhood $\Delta' \subset C$ of $c$ (which we may choose biholomorphic to the open unit disc), the set of holomorphic maps $\mathrm{Hol}((\Delta',c),(U,u))$ is relatively compact in $\mathrm{Hol}(\Delta, U)$ for some   open neighbourhood $\Delta \subset \Delta'$ of $c$.   
Since  $\mathrm{Hom}((C,c),(U,u))$ is a subset of $\mathrm{Hol}(\Delta,U)$ via restricting a morphism $f\colon C\to U$ to  a holomorphic map $\Delta\to U$, we conclude that the set $\mathrm{Hom}((C,c),(U,u))$   is relatively compact in $\mathrm{Hol}(\Delta, U)$.
\end{proof}

\begin{theorem}\label{thm:from_kob_to_pointed_rigidity}    Let $U$ be a smooth quasi-projective variety which is hyperbolic modulo  a proper (Zariski-)closed subset $\Sigma\subsetneq U$.  Let $u\in U\setminus \Sigma$.       Let $C$ be a smooth quasi-projective curve,  and let $c
\in C$. Then, every pointed morphism $(C,c)\to (U,u)$ is rigid.
\end{theorem} 

\begin{proof}  Combine  Theorem \ref{thm:rig_via_ac0} and Theorem \ref{thm:almost_compactness}.
\end{proof}

 \subsection{Characterizing relative compactness in terms of hyperbolicity} 
If $U$ is a   complex manifold, a  result often attributed to Royden and Kobayashi states that $d_U$ is the path-integrated form of an upper-semicontinuous complex Finsler pseudometric $F_U$ (often called the Kobayashi-Royden pseudometric) defined, for a tangent vector $\vec{e} \in T_u U$ at $u \in U$, by 
\[
F_U(\vec{e})=\frac 1 R, \ \ R=\sup\left\{\, |\zeta|\ \  \big{ |} \  \zeta\in\mathbb{C},  \  f'(0)=\zeta \vec{e},\ f\in \mathrm{Hol}((\bD,0), (U,u))\right\}.\] 

We now prove that the hyperbolicity of $U$ at a point $u$ is \emph{equivalent} to the relative compactness of the family of pointed maps from the open unit disc, after restriction to a suitably smaller disc.  (Note that in this result we need $U$ to be a \emph{smooth} complex space.)

\begin{theorem} \label{thm:equivs}  Let $U$ be a complex manifold embedded as a relatively compact subset of a hermitian complex manifold $(X,h)$ and let $u\in U$. Then the following are equivalent. \\[-3mm]
\begin{enumerate}
\item There is a neighbourhood $V_u$ of $u$ in $U$ and a bound $M>0$ such that for every pointed holomorphic map $f:(\bD,0) \to (U,v)$ with $v\in V_u$, we have $\Vert f'(0)\Vert_h < M.$  
\item  The Kobayashi-Royden pseudometric $F_U$ is uniformly nondegenerate at~$u$, i.e.,  there is a neighbourhood $V$ of $u$ in $U$ and   an $\epsilon > 0$ such that, for all $v\in V,$ and $ \vec{e}\in T_vU$,
\[
F_U(\vec{e})\geq \epsilon \Vert \vec{e}\Vert_h.
\]
\item The complex manifold $U$ is hyperbolic at $u$.  
\item If $\bD$ is the open unit disc, then there is a neighbourhood $\bD_0\subset \bD$ of $0$ such that the natural inclusion 
$$
\mathrm{Hol}((\bD,0), (U,u))\hookrightarrow \mathrm{Hol}(\bD_0, U), \ f\mapsto f|_{\bD_0}
$$
 maps $\mathrm{Hol}((\bD,0), (U,u))$ to a relatively compact subset of $\mathrm{Hol}(\bD_0, U)$ with respect to the compact-open topology.\\[-3mm]
\end{enumerate}
\end{theorem} 

\begin{proof}
First,  note that (1) $\iff$ (2) by the definition of $F_U$.
 If $(2)$ holds, then  $F_U$ is bounded below by $h$ in a neighborhood $V$ of $u$ and thus, for every $v$ in $V$ and $\vec{e} \neq 0$ in $T_vU$, we have that   $F_U(\vec{e}) > 0$. Since $d_U$ is the integrated form of $F_U$ on the   manifold $U$, we see that $d_U(v,w) > 0$ for $v \neq w$ in $V$. This implies $d_U$ is nondegenerate on $V$ and shows that $(2)\implies (3)$.  Since   $(3)\implies (4)$ by Theorem \ref{thm:nondeg_kob_implies_rel_compactness}, it suffices to show that $(4)\implies (1)$.

To show that $(4)\implies (1)$,  we suppose   that $(1)$ is false. Then there exists a sequence of points $v_i \to u$ in $U$ and (pointed) holomorphic maps $f_i : (\bD, 0) \to (U, v_i)$ such that  
\[
\Vert f_i'(0) \Vert_h \to \infty \quad \text{as } i \to \infty.
\]
Let $\bD_0 \subset \bD$ be as in condition (4). Since each $f_i \in \mathrm{Hol}((\bD,0), (U,u))$, restriction defines a sequence $f_i|_{\bD_0} $ in $ \mathrm{Hol}(\bD_0, U)$. By (4), the set of these restricted maps is relatively compact in the compact-open topology. Hence,  after passing to a subsequence, we may assume
$
f_i|_{\bD_0} \to \varphi $  uniformly on compact subsets of $ \bD_0$,  
for some holomorphic map $\varphi : \bD_0 \to U$.
Thus, the derivatives $f_i'(0)$ (computed in a chart and in the Hermitian metric $h$) must remain bounded. But by construction, $\Vert f_i'(0) \Vert_h \to \infty$, which contradicts the existence of such a convergent subsequence. This contradiction shows that (1) must hold.
\end{proof} 

\begin{remark}\label{rem:main_thm_kobayashi_details}  
A variant of Brody’s reparametrization technique, as used in \cite[Prop.~1.19]{Voisin}, can also be adapted to give the equivalence of (1) and (3) (even in the noncompact case).
\end{remark}

  \section{Finiteness results}\label{section:final} 
 In the previous section we always assumed the source of our maps to be a curve.  For proving   finiteness of pointed maps this is harmless, as the following simple lemma shows.
 
   \begin{lemma}[Reduction   to pointed curves]\label{lem:geom_hyp_on_curves}   Let $\Sigma$ be a proper closed subset of a variety $U$ and $u\in U\setminus \Sigma$. Suppose that, for every  smooth quasi-projective   pointed curve $(C,c)$,  the set of morphisms $f\colon C\to U$ with $f(c) = u$ is finite. Then, for every   pointed variety $(Y,y)$  the set of morphisms $f\colon Y\to U$ with $f(y)=u$ is finite.
 \end{lemma}
 \begin{proof} We argue by contrapositive.  Thus, suppose that there exists a variety  $Y$, a point $y\in Y$,  and pairwise distinct morphisms  $f_1,f_2,\ldots$ from $Y$ to $U$ which map $y$ to $u$. For $i$ and $j$ positive integers, we let $Y^{i,j}\subset Y$ be the closed subset    of points  $P$   such that $f_i(P) = f_j(P)$. Note that, for all $i\neq j$, the subset $Y^{i,j}$ is a proper closed subset of $Y$. Moreover, the subset $f_i^{-1}(\Sigma)$ is also a proper closed subset, as the image of $f_i$ contains the point $u$ and $u$ is not in $\Sigma$.   Since $\CC$ is uncountable, we may choose a (closed) point $w$ in  $$Y\setminus \cup_{i\neq j} Y^{i,j} \bigcup \cup_i f_i^{-1}(\Sigma). $$    Let $C$ be a smooth  quasi-projective  connected   curve and let $g\colon C\to Y$ be a morphism whose image contains   $w$ and $y$. Then the morphisms $f_1\circ g, f_2 \circ g, \ldots$ are   pairwise distinct morphisms from $C$ to $U$ and send $y$ to $u$. This concludes the proof.
 \end{proof}
 
 We now come to the proof of the main finiteness result for pseudohyperbolic varieties:
    
     \begin{proof}[Proof of Theorem \ref{thm:geometric_hyperbolicity_intro}]
     We may assume that $Y$ is a smooth quasi-projective curve (Lemma \ref{lem:geom_hyp_on_curves}). Now, since  $U$ is bounded modulo $\Sigma$,  the Hom scheme $\mathrm{Hom}(Y,U)$ is of finite type. It follows that $\mathrm{Hom}((Y,y),(U,u))$ is of finite type. Since $U$ is hyperbolic modulo $\Sigma$, the Hom scheme $\mathrm{Hom}((Y,y),(U,u))$ is zero-dimensional.  We conclude that $\mathrm{Hom}((Y,y),(U,u))$ is a finite type zero-dimensional scheme,  hence finite.
     \end{proof}

To prove the main finiteness result for base spaces (Theorem \ref{thm1}), we invoke the pseudohyperbolicity of base spaces of families with maximal variation in moduli: 
 
 \begin{theorem}  \label{thm:deng}
  Let $U$  be a   quasi-projective   variety.   If  there exists   a smooth projective family    $V\to U$ of varieties with semi-ample canonical bundle of maximal variation, then $U$ is pseudohyperbolic.
 \end{theorem} 

This is  proven by Deng \cite[Theorem~B]{Deng},  and is   the culmination of the work of  Viehweg-Zuo \cite{VZ03} (whose constructions are greatly inspired by \cite{Demailly} and \cite{LY}),  the works of Schumacher \cite{Schumacher2, Schumacher}, To-Yeung \cite{ToYeung},  M\"uller-Stach--Sheng--Yen--Zuo \cite{MSYZ},   Popa-Schnell \cite{Popa-Sch},  and Popa-Taji-Wu \cite{PopaTajiWu}.  
 
   \begin{proof}[Proof of Theorem \ref{thm1}]  By Theorem \ref{thm:deng} and Theorem \ref{thm:bounded}, the variety $U$ is pseudohyperbolic and pseudo-bounded. Thus, the desired finiteness result follows from Theorem \ref{thm:geometric_hyperbolicity_intro}.
   \end{proof}

 \section{Applications: Dimension bound and inheriting   hyperbolicity}
\label{section:apps}

The results of this paper, particularly the finiteness of pointed maps (Theorem \ref{thm:geometric_hyperbolicity_intro}), contribute to the broader understanding of hyperbolic geometry for quasi-projective varieties. Such results support the conjectural framework connecting various notions of hyperbolicity, often defined via complex analytic  tools like families of maps from the disc $\bD$ (see Section \ref{section:compactness}). This is encapsulated by the following extension of the Lang--Vojta conjectures.

\begin{conjecture_env} \label{conj:merged_LV} If $U$ is a    variety, then the following are equivalent.
\begin{enumerate}
\item The variety $U$ is   pseudohyperbolic  
\item There is a proper closed subset $\Sigma\subset U$ such  that every  entire  curve $\mathbb{C}\to U^{}$ factors over  $\Sigma^{}$.
\item   There is a proper closed subset $\Sigma \subset U$ such that, for every    pointed variety $(Y,y)$  and every $u\in U\setminus \Sigma $, the set  $\mathrm{Hom}((Y,y),(U,u))$ is finite.
\item The variety $U$ is of log-general type.
\item  The variety $U$ is pseudo-bounded.
\end{enumerate}
\end{conjecture_env}

A fundamental expectation is that $\mathrm{Hom}(C,U)$ should inherit hyperbolicity from $U$. The following remark illustrates this for Brody hyperbolicity.
 
\begin{remark}[Motivation]\label{remark:Brody}
Let $U$ be a Brody hyperbolic variety. Then, for every smooth quasi-projective curve $C$ and every irreducible subvariety $M \subset \mathrm{Hom}(C, U)$, the variety $M$ is also Brody hyperbolic. 
To see this, let 
$f \colon \mathbb{C} \to M^{}$ 
be a holomorphic map. For each point $c \in C$, consider the composition of $f$ with the analytification of the evaluation map
$\mathrm{ev}_c \colon M \to U.$
Since $U$ is Brody hyperbolic, this composition is constant for each $c$ in $C$, which readily implies that $f$ itself must be constant. This completes the proof.
\end{remark}

Rigidity of pointed maps to a pseudohyperbolic   variety (Theorem \ref{thm:from_kob_to_pointed_rigidity}) allows us to establish the inheritance property for  hyperbolic varieties, as well as the dimension drop. 

\begin{theorem}\label{thm:dimension0b}
Let $U$ be a variety and let $\Sigma\subsetneq U$ be a proper closed subset such that $U$ is   hyperbolic modulo $\Sigma$. Then,  for every smooth quasi-projective curve $C$ and  every irreducible subvariety   $M$ of $\mathrm{Hom}(C,U)$, the
  variety $M$ is   hyperbolic   modulo $\mathrm{Hom}(C,\Sigma)\cap M$. If in addition $U$ is of log-general type, then the Hom scheme $\mathrm{Hom}^{nc}(C,U)$   of non-constant maps from $C$ to $U$ satisfies
\[
\dim \left(\mathrm{Hom}^{nc}(C,U)\setminus \mathrm{Hom}(C,\Sigma)\right) \leq \dim U - 1.
\]
\end{theorem}
\begin{proof}  Since $M$ is of finite type, 
by Theorem \ref{thm:from_kob_to_pointed_rigidity}, for every $c$ in $C$, the evaluation $\mathrm{ev}_c:M\to U$ has zero-dimensional, and thus   finite,  fibres over $U\setminus \Sigma$. This implies that, for every $c$ in $C$, the variety $M$ is   hyperbolic   modulo the    closed subset $\mathrm{ev}_c^{-1}(\Sigma)$. In particular, $M$ is   hyperbolic modulo \[
  \bigcap_{c\in C} \mathrm{ev}_c^{-1}(\Sigma)=  \mathrm{Hom}(C,\Sigma)\cap M. \]
  
Now,  for every $c$ in $C$, we have the evaluation map $\mathrm{ev}_c:M\to U$ defined by $f\mapsto f(c)$.  Note that there exists a point $c$ in $C$ such that $\mathrm{ev}_c(M)\not\subset \Sigma$.  In particular, since   $\mathrm{ev}_c$ has finite fibres over $U\setminus \Sigma$ (Theorem \ref{thm:from_kob_to_pointed_rigidity}), it is generically finite onto its image, so that  $\dim M   \leq \dim U$.

 We  now show that $\dim M < \dim U$.  Since   $\dim M \leq \dim U$, it suffices to show that $\dim M \neq \dim U$.  To do so, we argue by contradiction.   
   
   Thus, suppose that $\dim M = \dim U$. Let $\Sigma_0$ be the set of $c$ in $C$ such that $\mathrm{ev}_c(M)\not\subset \Sigma$.  Note that $\Sigma_0$ is infinite.  Now,   for $c$ in $\Sigma_0$, as $\mathrm{ev}_c:M\to U$ is generically finite onto its image and $\dim M = \dim U$,     we have that $\mathrm{ev}_c:M\to U$ is dominant.
Since $U$ is of log-general type (by assumption),  by Tsushima's    finiteness theorem for varieties of log-general type \cite{Tsushima}, the set of dominant morphisms $M\to U$ is finite.   

In particular, the set    $\{\mathrm{ev}_c \ | \ c\in \Sigma_0\}$ is finite, as every $\mathrm{ev}_c$ is dominant for $c\in \Sigma_0$. This implies that there is an infinite (hence dense) subset $\Sigma_2\subset \Sigma_0\subset C$ such that, for every $c$ and $d$ in $\Sigma_2$, we have that $\mathrm{ev}_{c}=\mathrm{ev}_d$. This implies that   every $f$ in $M$ maps $C$ to a single point, and is thus constant, contradicting the fact that $M\subset \mathrm{Hom}^{nc}(C,U)$.    This   concludes the proof.
\end{proof}
  
\begin{remark}
Let $U$ be as in Theorem \ref{thm:dimension0b}. Then,  using the main result of \cite{DengLuSunZuo},  one can show that there is a proper closed subset $\Sigma\subsetneq U$ such that, for every smooth quasi-projective curve $C$ and every irreducible variety $M$ of $\mathrm{Hom}(C,U)$, the variety $M$ is ``Borel hyperbolic'' modulo    $\mathrm{Hom}(C,\Sigma)\cap M$.
\end{remark}

\subsection{Families with maximal variation}

We will apply Campana--P\u{a}un's theorem (formerly Viehweg's hyperbolicity conjecture) to prove inheritance properties and dimension bounds for the base space of a maximally varying family of polarized varieties.
 
\begin{theorem}[Campana--P\u{a}un] \label{thm:CP} 
Let  $U$ be a variety. Assume that  there exists a    smooth projective   family   $V\to U$  of varieties with semi-ample canonical bundle   which has maximal variation in moduli. Then  there is a proper closed subvariety $\Sigma\subsetneq U$ such that, for every closed subvariety $U'\subset U$ not contained in $\Sigma$, we have that $U'$ is of log-general type.
\end{theorem}
\begin{proof} Since $ V\to U$ has  maximal variation in moduli, there is a proper closed subset $\Sigma\subsetneq U$ such that, for every $u$ in $U\setminus \Sigma$, the Kodaira-Spencer map of $V\to U$ is injective at $u$.  In particular, if $U'\subset U$ is a subvariety of $U$ not contained in $\Sigma$, then the   restricted family $V\times_U U'\to U'$   is   a smooth projective family of varieties with semi-ample canonical bundle   which (still) has     maximal variation in moduli.  Thus, to prove the theorem, 	 it suffices to show that $U$ is of log-general type. To do so,  we may and do assume that $U$ is smooth.  Choose a smooth projective compactification $X$ of $U$ with boundary divisor $D=X\setminus U$ a simple normal crossings divisor.  Now,  the result follows from Campana--P\u{a}un's theorem \cite[Corollary~4.6]{CP} if the fibres are canonically polarized. However,  as noted in  \cite[Section~A.2]{Popa-Sch}, the  proof of Campana--P\u{a}un also works when the fibres have semi-ample canonical bundle.  More precisely,  by  \cite[Theorem~1.4.(iii)]{VZ02},   some tensor power of $\Omega^1_X(\log D)$  contains a subsheaf with big determinant.   In particular, by \cite[Theorem~4.1]{CP} (or \cite[Theorem~1]{Sch17}), the divisor $K_X+D$ is big, so that the smooth variety $U$ is of log-general type.
\end{proof}

\begin{theorem}\label{thm:inheritance_main2}    
Let $U$ be a variety. Assume that there is a     smooth projective family   $V \to U$ of varieties with semi-ample canonical bundle which has maximal variation in moduli.  Then, there is a proper closed subset $\Sigma\subsetneq U$ such that, for every smooth curve $C$ the following statements hold. 
\begin{enumerate}
\item For  every irreducible subvariety   $M$ of $\mathrm{Hom}(C,U)$, any subvariety of $M$ not contained in   $\mathrm{Hom}(C,\Sigma)\cap M$ is of log-general type and  
    $M$ is   hyperbolic   modulo $\mathrm{Hom}(C,\Sigma)\cap M$.
    \item We have that   \[
\dim \left(\mathrm{Hom}^{nc}(C,U)\setminus \mathrm{Hom}(C,\Sigma)\right) \leq \dim U - 1.
\] 
    \end{enumerate} 
\end{theorem}
\begin{proof}  By Theorem \ref{thm:deng}, we may choose $\Sigma\subsetneq U$  such that $U$ is hyperbolic modulo $\Sigma$.  The fact that $M$ is  hyperbolic   modulo $\mathrm{Hom}(C,\Sigma)\cap M$ then follows from Theorem \ref{thm:dimension0b}.

 Now, by Theorem \ref{thm:CP},   replacing $\Sigma$ by a larger proper closed subset if necessary, every subvariety of $U$ not contained in $\Sigma$ is of log-general type.  By Theorem \ref{thm1} (or Theorem \ref{thm:from_kob_to_pointed_rigidity}),  replacing $\Sigma$ by a larger proper closed subset if necessary,  for every  smooth curve $C$ and every $c$ in $C$, the evaluation map $\mathrm{ev}_c:M\to U$ is quasi-finite over $U\setminus \Sigma$. In particular, for every $c$ in $C$, every subvariety of $M$ not contained in    $\mathrm{ev}_c^{-1}(\Sigma)$ is of log-general type.  Thus, every subvariety of $M$ not contained in  $
 \bigcap_{c\in C} \mathrm{ev}_c^{-1}(\Sigma) =  \mathrm{Hom}(C,\Sigma)\cap M$ is of log-general type, as required.
 
 Finally,   since $U$ is of log-general type (Theorem \ref{thm:CP}), the dimension inequality  follows from     Theorem \ref{thm:dimension0b}.
\end{proof}

  If a family  of canonically polarized varieties $ V\to U$ has maximal variation in moduli, the associated moduli map  can have positive-dimensional fibres (consider a blowing-up of $U$), so  that one can not expect a bound on the dimension of $\mathrm{Hom}^{nc}(C,U)$ without excluding maps into $\Sigma$. That is, one can only bound the dimension of $\mathrm{Hom}^{nc}(C,U)\setminus \mathrm{Hom}(C,\Sigma)$.
     Indeed, the exceptional locus $\Sigma$ could very well contain a compact rational curve (in which case $\mathrm{Hom}(\mathbb{P}^1,\Sigma)$ has infinitely many connected components of unbounded dimension).

\begin{proof}[Proof of Theorem \ref{thm:dimension_drop_intro}] This follows from induction on $\dim U$ and Theorem \ref{thm:inheritance_main2}.    
\end{proof}

\subsection{Families with quasi-finite moduli map and stronger results}
 
Under stronger hypotheses, such as for effectively parametrized families, the finiteness of pointed maps (Theorem \ref{thm1}) holds without an exceptional set $\Sigma$.
 In \cite[Theorem~C]{Deng} (see also \cite{ToYeung}),  Deng proves the   hyperbolicity of the base space $U$ of a smooth projective family $V\to U$ of     smooth projective varieties  with ample canonical bundle, assuming that $U$ is smooth and the  family is \emph{effectively parametrized}, i.e., for every point $u$ in $U$, the Kodaira-Spencer map
 \[\rho_u\colon 	T_{U,y}\to \mathrm{H}^1(V_u, T_{V_u})\] is injective.  
 
\begin{corollary} 
 Let $U$ be a smooth variety such that there exists an effectively parametrized smooth projective family $V\to U$ of     varieties with ample canonical bundle.Then, for every pointed variety $(Y,y)$ and    $u$ in $U$, the set of pointed morphisms $ (Y,y)\to (U,u)$  is finite.   
\end{corollary}
\begin{proof}  To show that $U$ is bounded (see Definition \ref{def:bounded}), it suffices to show that every subvariety of $U$ is pseudo-bounded.  If $U'$ is a subvariety of $U$,  let $U''\to U'$ be a resolution of singularities. Since the moduli map $U\to \mathcal{CP}^c$ of the family $V\to U$ to the coarse moduli space $\mathcal{CP}^c$ of canonically polarized smooth proper varieties is quasi-finite (as the family is effectively parametrized),  the family $V\to U$ induces a family $V''\to U''$ of maximal variation in moduli.  Thus, $U''$ is pseudo-bounded by Theorem \ref{thm:bounded}, so that $U'$ is pseudo-bounded by the  birational invariance of pseudo-boundedness.
 Since $U$ is hyperbolic \cite[Theorem~C]{Deng},  we see that   Theorem \ref{thm:geometric_hyperbolicity_intro}   applies with $\Sigma=\emptyset$. 
\end{proof}

\begin{remark}\label{remark:murphy}
The assumption that $U$ is a smooth variety \emph{and} effectively parametrizes some polarized family is quite restrictive. Indeed, the singularities of the moduli stack of canonically polarized varieties satisfy Murphy's law \cite{Vakil}; more precisely, every possible type of singularity appears on the stack of smooth proper canonically polarized varieties. Consequently, there exist components of this stack that do not admit any unramified surjective morphism from a smooth variety.
\end{remark}

 \bibliography{refsci_new}{}

\def\cprime{$'$}
\begin{thebibliography}{MSSYZ15}

\bibitem[AT]{AscherTaji}
K.~Ascher and B.~Taji.
\newblock Boundedness results for families of non-canonically polarized
  projective varieties.
\newblock {\em arXiv:2408.15153}.

\bibitem[BJK24]{vBJK}
R.~van Bommel, A.~Javanpeykar, and L.~Kamenova.
\newblock Boundedness in families with applications to arithmetic
  hyperbolicity.
\newblock {\em J. Lond. Math. Soc. (2)}, 109(1):Paper No. e12847, 51, 2024.

\bibitem[BJR23]{BJR}
F.~Bartsch, A.~Javanpeykar, and E.~Rousseau.
\newblock Weakly-special threefolds and non-density of rational points.
\newblock 2023.
\newblock arXiv:2310.09065.

\bibitem[CHSZ]{Chenetal}
K.~Chen, T.~Hu, R.~Sun, and K.~Zuo.
\newblock On the distribution of non-rigid families in the moduli spaces.
\newblock {\em arXiv:2408.11604}.

\bibitem[CP15]{CP}
F.~Campana and M.~P\u{a}un.
\newblock Orbifold generic semi-positivity: an application to families of
  canonically polarized manifolds.
\newblock {\em Ann. Inst. Fourier (Grenoble)}, 65(2):835--861, 2015.

\bibitem[Deb01]{Debarrebook}
O.~Debarre.
\newblock {\em Higher-dimensional algebraic geometry}.
\newblock Universitext. Springer-Verlag, New York, 2001.

\bibitem[Dem97]{Demailly}
J.-P. Demailly.
\newblock Algebraic criteria for {K}obayashi hyperbolic projective varieties
  and jet differentials.
\newblock In {\em Algebraic geometry---{S}anta {C}ruz 1995}, volume~62 of {\em
  Proc. Sympos. Pure Math.}, pages 285--360. Amer. Math. Soc., Providence, RI,
  1997.

\bibitem[Den22]{Deng}
Y.~Deng.
\newblock On the hyperbolicity of base spaces for maximally variational
  families of smooth projective varieties.
\newblock {\em J. Eur. Math. Soc. (JEMS)}, 24(7):2315--2359, 2022.
\newblock With an appendix by Dan Abramovich.

\bibitem[DLSZ24]{DengLuSunZuo}
Y.~Deng, S.~Lu, R.~Sun, and K.~Zuo.
\newblock Picard theorems for moduli spaces of polarized varieties.
\newblock {\em Math. Ann.}, 390(1):1125--1154, 2024.

\bibitem[Fal83]{FaltingsArakelov}
G.~Faltings.
\newblock Arakelov's theorem for abelian varieties.
\newblock {\em Invent. Math.}, 73(3):337--347, 1983.

\bibitem[Gro63]{SGA1}
A.~Grothendieck.
\newblock {\em Rev\^etements \'etales et groupe fondamental (SGA I) {F}asc.
  {II}: {E}xpos\'es 6, 8 \`a 11}, volume 1960/61 of {\em S\'eminaire de
  G\'eom\'etrie Alg\'ebrique}.
\newblock Institut des Hautes \'Etudes Scientifiques, Paris, 1963.

\bibitem[JK20]{JKa}
A.~Javanpeykar and L.~Kamenova.
\newblock Demailly's notion of algebraic hyperbolicity: geometricity,
  boundedness, moduli of maps.
\newblock {\em Math. Z.}, 296(3-4):1645--1672, 2020.

\bibitem[JL24]{JLitt}
A.~Javanpeykar and D.~Litt.
\newblock Integral points on algebraic subvarieties of period domains: from
  number fields to finitely generated fields.
\newblock {\em Manuscripta Math.}, 173(1-2):23--44, 2024.

\bibitem[JSZ]{JSZ}
A.~Javanpeykar, R.~Sun, and K.~Zuo.
\newblock The {S}hafarevich conjecture revisited: Finiteness of pointed
  families of polarized varieties.

\bibitem[JV21]{JV}
A.~Javanpeykar and A.~Vezzani.
\newblock Non-archimedean hyperbolicity and applications.
\newblock {\em J. Reine Angew. Math.}, 778:1--29, 2021.

\bibitem[KL11]{KovacsLieblich}
S.J. Kov{\'a}cs and M.~Lieblich.
\newblock Erratum for {B}oundedness of families of canonically polarized
  manifolds: a higher dimensional analogue of {S}hafarevich's conjecture.
\newblock {\em Ann. of Math. (2)}, 173(1):585--617, 2011.

\bibitem[KM99]{KeelMcKernan}
S.~Keel and J.~McKernan.
\newblock Rational curves on quasi-projective surfaces.
\newblock {\em Mem. Amer. Math. Soc.}, 140(669):viii+153, 1999.

\bibitem[Kob98]{KobayashiBook}
S.~Kobayashi.
\newblock {\em Hyperbolic complex spaces}, volume 318 of {\em Grundlehren der
  Mathematischen Wissenschaften [Fundamental Principles of Mathematical
  Sciences]}.
\newblock Springer-Verlag, Berlin, 1998.

\bibitem[Kov03]{Kovacs}
S.~J. Kov{\'a}cs.
\newblock Families of varieties of general type: the {S}hafarevich conjecture
  and related problems.
\newblock In {\em Higher dimensional varieties and rational points ({B}udapest,
  2001)}, volume~12 of {\em Bolyai Soc. Math. Stud.}, pages 133--167. Springer,
  Berlin, 2003.

\bibitem[Kov05]{KovacsStrong}
S.~J. Kov\'{a}cs.
\newblock Strong non-isotriviality and rigidity.
\newblock In {\em Recent progress in arithmetic and algebraic geometry}, volume
  386 of {\em Contemp. Math.}, pages 47--55. Amer. Math. Soc., Providence, RI,
  2005.

\bibitem[LY90]{LY}
S.~S.-Y. Lu and S.-T. Yau.
\newblock Holomorphic curves in surfaces of general type.
\newblock {\em Proc. Nat. Acad. Sci. U.S.A.}, 87(1):80--82, 1990.

\bibitem[MSSYZ15]{MSYZ}
S.~M\"{u}ller-Stach, M.~Sheng, X.~Ye, and K.~Zuo.
\newblock On the cohomology groups of local systems over {H}ilbert modular
  varieties via {H}iggs bundles.
\newblock {\em Amer. J. Math.}, 137(1):1--35, 2015.

\bibitem[Mum08]{MumAb}
D.~Mumford.
\newblock {\em Abelian varieties}, volume~5 of {\em Tata Institute of
  Fundamental Research Studies in Mathematics}.
\newblock Published for the Tata Institute of Fundamental Research, Bombay; by
  Hindustan Book Agency, New Delhi, 2008.

\bibitem[Pat12]{Patakfalvi}
Z.~Patakfalvi.
\newblock Viehweg's hyperbolicity conjecture is true over compact bases.
\newblock {\em Adv. Math.}, 229(3):1640--1642, 2012.

\bibitem[PS17]{Popa-Sch}
M.~Popa and C.~Schnell.
\newblock Viehweg's hyperbolicity conjecture for families with maximal
  variation.
\newblock {\em Invent. Math.}, 208(3):677--713, 2017.

\bibitem[PTW19]{PopaTajiWu}
M.~Popa, B.~Taji, and L.~Wu.
\newblock Brody hyperbolicity of base spaces of certain families of varieties.
\newblock {\em Algebra Number Theory}, 13(9):2205--2242, 2019.

\bibitem[Sch12]{Schumacher2}
G.~Schumacher.
\newblock Positivity of relative canonical bundles and applications.
\newblock {\em Invent. Math.}, 190(1):1--56, 2012.

\bibitem[Sch14]{Schumacher}
G.~Schumacher.
\newblock Curvature properties for moduli of canonically polarized
  manifolds---an analogy to moduli of {C}alabi-{Y}au manifolds.
\newblock {\em C. R. Math. Acad. Sci. Paris}, 352(10):835--840, 2014.

\bibitem[Sch17]{Sch17}
C.~Schnell.
\newblock On a theorem of {C}ampana and {P}\u{a}un.
\newblock {\em \'{E}pijournal Geom. Alg\'{e}brique}, 1:Art. 8, 9, 2017.

\bibitem[SZ91]{SaitoZucker}
M.-H. Saito and S.~Zucker.
\newblock Classification of nonrigid families of {$K3$} surfaces and a
  finiteness theorem of {A}rakelov type.
\newblock {\em Math. Ann.}, 289(1):1--31, 1991.

\bibitem[Tsu79]{Tsushima}
R.i Tsushima.
\newblock Rational maps to varieties of hyperbolic type.
\newblock {\em Proc. Japan Acad. Ser. A Math. Sci.}, 55(3):95--100, 1979.

\bibitem[TY15]{ToYeung}
W.-K. To and S.-K. Yeung.
\newblock Finsler metrics and {K}obayashi hyperbolicity of the moduli spaces of
  canonically polarized manifolds.
\newblock {\em Ann. of Math. (2)}, 181(2):547--586, 2015.

\bibitem[Ura79]{Urata}
T.~Urata.
\newblock Holomorphic mappings into taut complex analytic spaces.
\newblock {\em T\^{o}hoku Math. J. (2)}, 31(3):349--353, 1979.

\bibitem[Vak06]{Vakil}
R.~Vakil.
\newblock Murphy's law in algebraic geometry: badly-behaved deformation spaces.
\newblock {\em Invent. Math.}, 164(3):569--590, 2006.

\bibitem[Vie83]{Vie83}
E.~Viehweg.
\newblock Weak positivity and the additivity of the {K}odaira dimension for
  certain fibre spaces.
\newblock In {\em Algebraic varieties and analytic varieties ({T}okyo, 1981)},
  volume~1 of {\em Adv. Stud. Pure Math.}, pages 329--353. North-Holland,
  Amsterdam, 1983.

\bibitem[Vie95]{ViehwegBook}
E.~Viehweg.
\newblock {\em Quasi-projective moduli for polarized manifolds}, volume~30 of
  {\em Ergebnisse der Mathematik und ihrer Grenzgebiete (3)}.
\newblock Springer-Verlag, Berlin, 1995.

\bibitem[Voi03]{Voisin}
C.~Voisin.
\newblock On some problems of {K}obayashi and {L}ang; algebraic approaches.
\newblock In {\em Current developments in mathematics, 2003}, pages 53--125.
  Int. Press, Somerville, MA, 2003.

\bibitem[VZ02]{VZ02}
E.~Viehweg and K.~Zuo.
\newblock Base spaces of non-isotrivial families of smooth minimal models.
\newblock In {\em Complex geometry ({G}\"{o}ttingen, 2000)}, pages 279--328.
  Springer, Berlin, 2002.

\bibitem[VZ03]{VZ03}
E.~Viehweg and K.~Zuo.
\newblock On the {B}rody hyperbolicity of moduli spaces for canonically
  polarized manifolds.
\newblock {\em Duke Math. J.}, 118(1):103--150, 2003.

\bibitem[Yam19]{YamanoiKob}
K.~Yamanoi.
\newblock Pseudo {K}obayashi hyperbolicity of subvarieties of general type on
  abelian varieties.
\newblock {\em J. Math. Soc. Japan}, 71(1):259--298, 2019.

\bibitem[Zuo00]{ZuoNeg}
K.~Zuo.
\newblock On the negativity of kernels of {K}odaira-{S}pencer maps on {H}odge
  bundles and applications.
\newblock {\em Asian J. Math.}, 4(1):279--301, 2000.

\end{thebibliography}
\bibliographystyle{alpha}

\end{document}